\documentclass[11pt]{amsart}

\usepackage{amsmath, amsthm, amsfonts, amssymb, latexsym, graphicx, dcpic, pictex}

\newtheorem{thm}{Theorem}[section]
\newtheorem{lemma}[thm]{Lemma}
\newtheorem{prop}[thm]{Proposition}
\newtheorem{cor}[thm]{Corollary}

\newtheorem*{unno_thm}{Theorem~\ref{main_thm}}

\theoremstyle{definition}
\newtheorem{defn}[thm]{Definition}
\newtheorem{example}[thm]{Example}
\newtheorem{conv}[thm]{Convention}
\newtheorem{bold_notation}[thm]{Notation}

\newtheorem*{rem}{Remark}

\DeclareMathOperator*{\colim}{colim}
\DeclareMathOperator*{\hocolim}{hocolim}

\newcommand{\ds}{\displaystyle}
\newcommand{\mcA}{\mathcal{A}}
\newcommand{\mcB}{\mathcal{B}}
\newcommand{\mcC}{\mathcal{C}}
\newcommand{\mcD}{\mathcal{D}}
\newcommand{\mcE}{\mathcal{E}}
\newcommand{\mcI}{\mathcal{I}}
\newcommand{\mcL}{\mathcal{L}}
\newcommand{\mcM}{\mathcal{M}}
\newcommand{\mcO}{\mathcal{O}}
\newcommand{\mcP}{\mathcal{P}}
\newcommand{\mcQ}{\mathcal{Q}}
\newcommand{\mcR}{\mathcal{R}}
\newcommand{\mcS}{\mathcal{S}}
\newcommand{\mcU}{\mathcal{U}}
\newcommand{\op}{\mathrm{op}}
\newcommand{\Iso}{\mathrm{Iso}}
\newcommand{\Sets}{\mathrm{Sets}}
\newcommand{\Hom}{\mathrm{Hom}}
\newcommand{\ho}{\mathrm{ho}}

\newcommand{\bfn}{\mathbf{n}}
\newcommand{\kplus}{\mathbf{k}_+}
\newcommand{\mplus}{\mathbf{m}_+}
\newcommand{\nplus}{\mathbf{n}_+}
\newcommand{\rplus}{\mathbf{r}_+}
\newcommand{\splus}{\mathbf{s}_+}
\newcommand{\ilcolim}{\mathrm{colim}}
\newcommand{\ilhocolim}{\mathrm{hocolim}}
\newcommand{\sk}{\mathrm{sk}}
\newcommand{\dom}{\mathrm{dom}}

\begin{document}

\title[Model category extensions of Pirashvili-S{\l}omi\'{n}ska theorems]{Model category extensions of the Pirashvili-S{\l}omi\'{n}ska theorems}

\author{Randall D. Helmstutler}

\address{Department of Mathematics\\
         University of Mary Washington\\
         1301 College Avenue\\
         Fredericksburg, VA  22401}

\email{rhelmstu@umw.edu}



\date{13 May 2008}

\begin{abstract}
\noindent  We describe the class of \emph{semi-stable} model categories, which generalize the equivalence of finite products and coproducts
in abelian and stable model categories, and use this to establish Morita equivalences among categories
of functors.  We provide a construction of pairs of small categories---known as \emph{conjugate pairs}---whose
associated categories of diagrams are Quillen equivalent in the semi-stable setting.  We frame our development in the context of Morita theory, following S{\l}omi\'{n}ska's work on
similar questions for categories of functors enriched over and taking values in $R$-modules.
\end{abstract}

\maketitle


\section{Introduction}\label{intro_sect}

\subsection{Background}

Given a model category $\mcC$ and any small category $\mcD$, one
can form the category $[\mcD, \mcC]$ of functors
from $\mcD$ to $\mcC$.  When $\mcC$ is cofibrantly generated, this
category of diagrams inherits a model structure.  A basic question is
then to determine when two such model categories of diagrams in
$\mcC$ are Quillen equivalent.  In this article we develop an
approach to this problem for the case when $\mcC$ shares characteristics common to
abelian and stable model categories.  Precisely, in both types of categories the natural
map $X \vee Y \rightarrow X \times Y$ is an equivalence of the right kind:  an isomorphism
in the abelian case and a weak equivalence in the stable case.  Suitably generalized, this leads
to the notion of a \emph{semi-stable model category}.

The ideas of classical Morita theory have been used for some time
now in describing equivalences of categories of functors; an early example may be found in \cite{DCN}.  More recently,
S{\l}omi\'{n}ska has developed a robust set of sufficient conditions that yield Morita equivalences for
categories of functors taking values in $R$-modules \cite{JS}.  Up to this point, such work has been primarily
an algebraist's affair.  Our approach is to deduce a Morita theorem for diagrams in a
semi-stable model category by suitably adapting a portion of S{\l}omi\'{n}ska's framework.

The genesis of both S{\l}omi\'{n}ska's original work and our
current homotopy-theoretic version may be found in the results of
Pirashvili in \cite{TP}.  Let $\mcE$ denote the category with
objects the sets $\bfn = \{1, 2, \ldots, n\}$ and with the
surjective maps as the morphisms (where $\mathbf{0}$ is the empty
set).  Let $\Gamma$ denote the category with objects the finite
based sets $\nplus = \{0, 1, 2, \ldots, n\}$ and morphisms the
based maps, where $0$ acts as the basepoint. Paraphrased,
Pirashvili proves in \cite{TP} that the functor categories
$[\Gamma^\op, \mcC]$ and $[\mcE^\op, \mcC]$ are equivalent when
$\mcC$ is an abelian category.  The equivalence is given by the
cross effect construction of Eilenberg and MacLane in \cite{EMcLII}.

At about the same time, similar equivalences were being noticed in stable homotopy theory, again arising from a (homotopical)
cross effect.  The categories $\mcE$ and $\Gamma$ appear together in several instances, such as in \cite{A1} and \cite{BM}.
In his work on generalizations of the infinite symmetric product to categories of $S$-algebras, Kuhn \cite{K1} explicitly
mentions an apparent connection to Pirashvili's algebraic equivalence of categories.  All of this suggests that the categories $[\Gamma^\op, \mcC]$ and $[\mcE^\op, \mcC]$ are Quillen equivalent when $\mcC$
is a stable model category.  This is indeed the case, as proven and generalized in the author's doctoral dissertation (independent
of the work of S{\l}omi\'{n}ska).

At this point it is an attractive idea to attempt to unify the algebraic theorem
with the stable homotopy analogue.  This is the motivation for considering semi-stable model categories:  our
main result proves both theorems at once.  Hence en route to the proof of the stable homotopy result we recover the abelian
case as well, but we do so via model-categorical methods.  While
we would never maintain that the techniques of model categories are a more sensible approach to abelian category questions, the point
is that we can prove the stable homotopy version of the theorem while capturing the similarities between these disparate
classes of categories.

S{\l}omi\'{n}ska's Morita theory \cite{JS} is a generalization of the behavior exhibited by the categories $\mcE$ and $\Gamma$.
She gives conditions on pairs $(\mcB, \mcA)$ of small categories enriched over $R$-modules so that there is an adjoint
equivalence between $[\mcB^\op, R\mathrm{-mod}]$ and $[\mcA^\op, R\mathrm{-mod}]$.  Suitably
enriched, the pair $(\Gamma, \mcE)$ is an example of such a Morita equivalent pair of small categories.  Not only does this succeed in
recovering Pirashvili's theorem (at least for the category of $R$-modules), it also lends itself to numerous additional
examples.

In this article we give similar conditions on pairs $(\mcB, \mcA)$ of small categories so that the associated
functor categories $[\mcB^\op, \mcC]$ and $[\mcA^\op, \mcC]$ are Quillen equivalent when $\mcC$ is semi-stable.
Like S{\l}omi\'{n}ska's work, this is achieved by abstracting certain relations between the categories $\mcE$ and $\Gamma$, although
we take a different approach necessitated by the non-additive nature of model categories.  In S{\l}omi\'{n}ska's framework
all categories and functors were additive, and because of this we will necessarily lose a couple of her examples.

Fortunately, the main feature of the pair $(\Gamma, \mcE)$ to be generalized is not hard to describe.  In fact, it closely parallels the First Isomorphism Theorem
in group theory.  Every $\Gamma$-map $\gamma: \mplus \rightarrow \nplus$ admits a uniquely determined three-fold
factorization
\[
\begindc{0}[3]
\obj(-9,8)[1]{$\mplus$} \obj(9,8)[2]{$\nplus$} \obj(-9,-8)[3]{$\rplus$}
\obj(9,-8)[4]{$\splus$} \mor{1}{2}{$\gamma$} \mor{3}{4}{$\alpha$}
\mor{1}{3}{$q$}[\atright, \solidarrow]
\mor{4}{2}{$i$}[\atright, \solidarrow] 
\enddc
\]
where
\begin{itemize}
\item  $q$ is the quotient map that collapses the ``kernel" $\gamma^{-1}(0)$ of $\gamma$,
\item  $i$ represents the inclusion of the image of $\gamma$ into its codomain, and
\item  $\alpha$ is the unique epimorphism making the diagram commute.
\end{itemize}
Note that $\alpha$ has the additional property that only the basepoint is mapped to the basepoint.  Thus we may forget
the basepoints and view $\alpha$ as a morphism $\alpha: \mathbf{r} \rightarrow \mathbf{s}$ in $\mcE$.

Thus we see that every $\Gamma$-map gives rise to a uniquely determined map in $\mcE$, and we can keep track of this
assignment by recording the ``cokernel" $q$ and the inclusion map $i$.  Furthermore, note that the category of
inclusion maps is in some sense dual to the category of quotients.  Abstracting all of this structure leads to our notion of a
\textit{conjugate pair} of small categories, of which $(\Gamma, \mcE)$ is the prime example.  This is made precise in
Section~\ref{conj_pair_section}.

The primary goal of this paper is to show that every conjugate pair is in fact a Morita equivalent
pair of small categories in the semi-stable setting.  Paraphrasing to avoid some currently undefined notation, our main result
asserts the following:

\begin{unno_thm}
Suppose that $\mcC$ is a model category.  Every conjugate pair $(\mcB, \mcA)$ of small categories gives rise to
an adjoint pair
\[
\begindc{0}[3]
\obj(-12,0)[1]{$[\mcB^\op, \mcC]$}
\obj(12,0)[2]{$[\mcA^\op, \mcC]$}
\obj(-8,1)[3]{} \obj(8,1)[4]{} \obj(-8,-1)[5]{} \obj(8,-1)[6]{} \mor{3}{4}{$\mcL$} \mor{6}{5}{$\mcR$}
\enddc
\]
and this is a Quillen equivalence when $\mcC$ is semi-stable.
\end{unno_thm}

This article is a generalization of several results found in the author's doctoral dissertation, written under the
direction of Prof.\ Nicholas Kuhn at the University of Virginia.  We are grateful for all the helpful suggestions he has
offered.

\subsection{Organization of the paper}

In Section~\ref{Morita_subsection} we lay our foundation for Morita equivalence for categories of functors.  We will
take classical Morita theory as inspiration for both the terminology and the notation.  Section~\ref{mod_cat_subsection} sets
in stone some blanket technical assumptions we must impose on model categories for the purposes of this paper, and
cofibrant generation is quickly reviewed.  In Section~\ref{stable_subsection} we give the definition of a \emph{semi-stable
model category}.

In Section~\ref{index_cat_sect} we begin the development of the small category side of the story, where we carefully
define the notion of an \textit{indexing category}.  Such categories will play the role of the category of inclusions
in the three-fold factorizations.  This type of category will also index some fundamental (co)product decompositions,
hence the name.

Section~\ref{conj_pair_section} formalizes three-fold factorization to the notion of a \textit{conjugate pair} of small
categories; plenty of examples follow in Section~\ref{examples_section}.  The bulk of the hard work lies in Section~\ref{nat_section},
where we show that a conjugate pair allows for the creation of a special functor, called the \textit{regular bimodule}.  Here,
the term \textit{bimodule} means a functor of the form $P: \mcA^\op \times \mcB \rightarrow \Sets_*$ (we are using
$\Sets_*$ to denote the category of based sets).  The regular bimodule conveniently encodes many naturality properties which are necessary for the proof of our
theorem, and these are detailed in Section~\ref{nat_section}.

Every bimodule creates an adjoint pair between functor categories.  For a regular bimodule, the resulting right adjoint
admits a nice product decomposition; this is the content of Section~\ref{adjoint_section}.  In Section~\ref{free_section},
we analyze the behavior of free functors and their pushouts under this adjunction.  The remainder of the paper is devoted
to proving Theorem~\ref{main_thm}, our main result.


\section{Categorical Preliminaries}\label{cat_section}

\subsection{The Morita theory context}\label{Morita_subsection}

For this discussion, let us fix a pointed category $\mcC$ having all limits and colimits.  In the applications in this article,
$\mcC$ will always be a pointed model category.

The constructions involved in our Morita theory naturally arise from products and coproducts indexed by based sets.  As an
initial technicality we must be clear about how we will handle the basepoint.  Given a based set $S$ and an object $C$ of $\mcC$, the product
$\prod_{S} C$ will always stand for the ordinary product of copies of $C$, one copy for each non-basepoint element of $S$.  That is,
the factor corresponding to the basepoint of $S$ will always be taken to be the zero object of $\mcC$.  The same convention
applies to coproducts $\bigvee_S C$ indexed by based sets (as all of our categories will be pointed, we will be using the
wedge symbol $\vee$ for the coproduct).  It will be convenient to denote $\bigvee_S C$ by $C \otimes S$ on occasion.

Let us now fix a small category $\mcA$.  Given functors $G: \mcA^\op \rightarrow \mcC$ and $P: \mcA^\op \rightarrow \Sets_*$
we can form a new functor $\mcA^\op \times \mcA \rightarrow \mcC$ by the assignment
\[
(x,y) \mapsto \prod_{P(y)} G(x).
\]
We define $\Hom^\mcA(P,G)$ to be the end of this functor (see \cite{SM} for a refresher on ends, if necessary).  We are writing
$\mcA$ as a superscript to distinguish this from a set of morphisms in a category.  It is entirely formal that this
construction has all of the expected properties of a $\Hom$-like object:  functoriality of the expected variances, correct
behavior with (co)products, and a Yoneda lemma.  For example, if $P = \mcA(-,a)_+$ is a representable (think:  free)
functor, we have a natural isomorphism $\Hom^\mcA(P,G) = G(a)$.

Likewise, we can use coends of coproducts to define a tensor product of functors.  Fix a small category $\mcB$ and functors
$F: \mcB^\op \rightarrow \mcC$ and $Q: \mcB \rightarrow \Sets_*$.  From this we can form a functor $\mcB^\op \times \mcB
\rightarrow \mcC$ via
\[
(x,y) \mapsto F(x) \otimes Q(y),
\]
the coend of which we denote $F \otimes_\mcB Q$.  Again, this has all of the expected properties of a tensor product.

The starting point in classical Morita theory is the simple fact that every bimodule gives an adjoint pair between categories
of modules.  In our situation, every functor of the form $P: \mcA^\op \times \mcB \rightarrow \Sets_*$ will give rise
to an adjoint pair between categories of contravariant functors; we will therefore call such functors \textit{bimodules}.

To that end, fix a bimodule $P: \mcA^\op \times \mcB \rightarrow \Sets_*$.  Our $\Hom^\mcA(-,-)$ and $- \otimes_\mcB -$ constructions
yield an adjoint pair
\[
\begindc{0}[3]
\obj(-12,0)[1]{$[\mcB^\op, \mcC]$}
\obj(12,0)[2]{$[\mcA^\op, \mcC]$}
\obj(-8,1)[3]{} \obj(8,1)[4]{} \obj(-8,-1)[5]{} \obj(8,-1)[6]{} \mor{3}{4}{$\mcL$} \mor{6}{5}{$\mcR$}
\enddc
\]
described as follows.  Given $F: \mcB^\op \rightarrow \mcC$ we define the functor $\mcL F: \mcA^\op \rightarrow \mcC$
by
\[
\mcL F(a) = F \otimes_\mcB P(a,-).
\]
Dually, given $G: \mcA^\op \rightarrow \mcC$ we define $\mcR G: \mcB^\op \rightarrow
\mcC$ by
\[
\mcR G(b) = \Hom^\mcA(P(-,b),G).
\]
It is entirely formal that $(\mcL, \mcR)$ is
an adjoint pair.  All of the Quillen equivalences in our work will arise from this type of adjunction.

\subsection{Model categories}\label{mod_cat_subsection}

We will take as the root definition of \textit{model category} that which is set forth in Chapter 1 of Hovey's \textit{Model
Categories} \cite{MH}.  In particular, we will assume that the functorial factorizations are fixed as part of the underlying model
structure (others assume that such factorizations merely exist).  We will also be required to make some additional
technical assumptions.  For brevity's sake, we will bundle these assumptions into the term \textit{model category}.

\begin{conv}\label{conv}
In this paper, the term \textit{model category} shall always mean
a pointed, cofibrantly generated model category $\mcC$ as in
\cite{MH}, possessing the following additional properties:
\begin{itemize}
\item  $\mcC$ is proper.
\item  The projections $X \stackrel{p_X}{\longleftarrow} X \times Y \stackrel{p_Y}{\longrightarrow} Y$ are always fibrations.
\item  If $f: A \rightarrow B$ and $g: C \rightarrow D$ are acyclic fibrations then the map $f p_A \times
g p_C: A \times C \rightarrow B \times D$ is a weak equivalence.
\item  If $f: A \rightarrow B$ and $g: C \rightarrow D$ are acyclic cofibrations then the map
$i_Bf \vee i_Dg: A \vee C \rightarrow B \vee D$ is a weak equivalence.
\end{itemize}
\end{conv}

For the most part these assumptions are fairly mild.  The properness assumption will guarantee that we have well-behaved homotopy
pullbacks and pushouts, as in \cite[Chapter 13]{PH}.  Note that all projections
will be fibrations whenever all objects in $\mcC$ are fibrant.  We also gain the following two results which will be
needed in Section~\ref{adjoint_section}.

\begin{prop}\label{prod_fib}
Whenever $f: A \rightarrow B$ and $g: C \rightarrow D$ are fibrations, so is the map $fp_A \times
gp_C: A \times C \rightarrow B \times D$.
\end{prop}

\begin{proof}
Since $p_A$ and $p_C$ are fibrations, so are $fp_A$ and $gp_C$.  The result follows since products of fibrations are
fibrations.
\end{proof}

\begin{prop}\label{prod_weq}
Let $\mcC$ be a model category in which the natural map $X \vee Y \rightarrow X \times Y$ is a weak equivalence
for all cofibrant objects $X$ and $Y$.  Suppose that $A$ and $C$ are cofibrant objects.  Then maps $f: A \rightarrow B$ and $g: C \rightarrow D$ are weak equivalences if
and only if $fp_A \times gp_C: A \times C \rightarrow B \times D$ is a weak equivalence.
\end{prop}

\begin{proof}
The ``if" direction ($\Leftarrow$) follows directly from the retract axiom in a general model category.  For the other
implication, suppose that $f$ and $g$ are weak equivalences.  We may factor $f$ and $g$ as cofibrations
followed by acyclic fibrations, say
\[
A \stackrel{f'}{\hookrightarrow} B' \stackrel{\sim}{\twoheadrightarrow} B
\]
and
\[
C \stackrel{g'}{\hookrightarrow} D' \stackrel{\sim}{\twoheadrightarrow} D
\]
respectively.  Note that $f'$ and $g'$ are necessarily weak equivalences, and that $B'$ and $D'$ are cofibrant.

We now have a commutative diagram
\[
\begindc{0}[3]
\obj(-12,8)[1]{$A \vee C$} \obj(12,8)[2]{$B' \vee D'$}
\obj(-12,-8)[3]{$A \times C$} \obj(12,-8)[4]{$B' \times D'$}
\obj(0,-20)[5]{$B \times D$} \mor{1}{2}{$(1)$}
\mor{1}{3}{$(2)$}[\atright, \solidarrow] \mor{2}{4}{$(3)$}
\mor{3}{5}{$(4)$}[\atright, \solidarrow] \mor{4}{5}{$(5)$}
\enddc
\]
in which map (4) is the map in question.  Maps (1) and (5) are weak equivalences by our conventions, while maps (2)
and (3) are weak equivalences by hypothesis.  Hence map (4) is a weak equivalence, as desired.
\end{proof}

We close this section with a few remarks about cofibrant generation and categories of functors; see \cite[Chapter 11]{PH} for complete
details.  When $\mcC$ is a cofibrantly generated model category,
the category $[\mcD^\op, \mcC]$ of functors $\mcD^\op \rightarrow \mcC$ inherits a model structure (here $\mcD$ can be any
small category).  This model structure is also cofibrantly generated, and weak equivalences and fibrations of
diagrams are defined objectwise.  We will need a complete description of the cofibrations.

\begin{defn}
Fix an object $d$ of $\mcD$ and an object $C$ of $\mcC$.  The \textit{free functor} generated by $d$ and $C$ is the
functor $F_d^C: \mcD^\op \rightarrow \mcC$ given by
\[
F_d^C(x) = C \otimes \mcD(x,d)_+.
\]
\end{defn}

Clearly every map $i: B \rightarrow C$ in $\mcC$ induces a natural transformation $i_*: F_d^B \rightarrow F_d^C$ of free
functors.  When $\mcC$ is cofibrantly generated, the cofibrations in $[\mcD^\op, \mcC]$ are generated by the maps of the
form $i_*: F_d^B \rightarrow F_d^C$, where $i$ is a generating cofibration of $\mcC$.  Thus all cofibrations in the
category of diagrams are obtained through transfinite compositions of pushouts of such maps (and retracts thereof).

In order to handle these transfinite compositions, we require one more technical tool before we proceed, namely a reasonable
sequential homotopy colimit functor.  Let $\mcC$ be a model category and fix an ordinal $\lambda$.  By a \textit{$\lambda$-sequence} in $\mcC$
we mean a functor $X: \lambda \rightarrow \mcC$, where $\lambda$ is made into a category in the usual way.  Since $\mcC$ is
cofibrantly generated, the category $[\lambda, \mcC]$ of all such sequences inherits a model structure with objectwise
weak equivalences and fibrations.  This is even true if $\mcC$ is not cofibrantly generated, though we will not need this
fact.  In this model structure, a $\lambda$-sequence is cofibrant if it is objectwise cofibrant and each map in the
diagram is a cofibration; see \cite[Example 4.3]{CS}.

We will use the letter $Q$ to denote cofibrant replacement.  According to Hirschhorn in
\cite[Proposition 17.9.1]{PH}, the map $\colim(g): \colim(X) \rightarrow \colim(Y)$ is a weak equivalence whenever
$g: X \rightarrow Y$ is a weak equivalence of cofibrant $\lambda$-sequences.  This of course implies that the natural map
$\colim(QX) \rightarrow \colim(X)$ is a weak equivalence whenever $X$ is a cofibrant $\lambda$-sequence.

One would like to define the homotopy colimit of a $\lambda$-sequence $X$ as the colimit of its cofibrant replacement.  This
turns out to be fine for objects, but some care must be taken with the morphisms.  We remark that cofibrant replacements of
objects are well-defined since the functorial factorizations are fixed as part of the model structure.  However, cofibrant replacements
of \textit{maps} are unique only up to (left) homotopy.  By passing to the homotopy category $\ho(\mcC)$, this problem
is remedied.

\begin{prop}\label{hocolim}
Forming colimits of cofibrant replacements yields a well-defined functor $\hocolim: [\lambda, \mcC] \rightarrow \ho(\mcC)$.  If $X$ is a cofibrant $\lambda$-sequence,
the natural map $\hocolim(X) \rightarrow \colim(X)$ is an isomorphism in the homotopy category.
\end{prop}


\subsection{Semi-stable model categories}\label{stable_subsection}

We can now describe the properties model categories must have in order to develop our Morita equivalences.  It is
helpful to keep in mind that each axiom below is a statement about the equivalence of finite products and coproducts.

\begin{defn}
We say that a model category $\mcC$ is \emph{semi-stable} if the following product-coproduct coherence axioms are satisfied:
\begin{itemize}
\item  \textbf{The Lower Triangular Axiom:}  Suppose that $A_1$ and $A_2$ are cofibrant objects and that the map $f: A_1 \vee A_2 \rightarrow
A_1 \times A_2$ has components $f_{ij}: A_i \rightarrow A_j$.  If the diagonal components $f_{ii}$ are weak equivalences
and $f_{12} = 0$, then $f$ is a weak equivalence.
\item \textbf{Pushout-Product Coherence:}  Given two homotopy pushout diagrams ($i=1,2$)
\[
\begindc{0}[3]
\obj(-9,8)[a]{$X_i$}
\obj(9,8)[b]{$Y_i$}
\obj(-9,-8)[c]{$Z_i$}
\obj(9,-8)[d]{$P_i$}
\mor{a}{b}{} \mor{c}{d}{}
\mor{b}{d}{} \mor{a}{c}{}
\enddc
\]
the product square
\[
\begindc{0}[3]
\obj(-11,8)[a]{$X_1 \times X_2$}
\obj(11,8)[b]{$Y_1 \times Y_2$}
\obj(-11,-8)[c]{$Z_1 \times Z_2$}
\obj(11,-8)[d]{$P_1 \times P_2$}
\mor{a}{b}{} \mor{c}{d}{}
\mor{b}{d}{} \mor{a}{c}{}
\enddc
\]
is also a homotopy pushout.
\item \textbf{Colimit-Product Coherence:}  Whenever $X$ and $Y$ are cofibrant $\lambda$-sequences in $\mcC$, the natural map
\[
\hocolim(X \times Y) \longrightarrow \hocolim(X) \times \hocolim(Y)
\]
is an isomorphism in the homotopy category.
\end{itemize}
\end{defn}

Note that these axioms are definitely one-sided and not self-dual.  This is intentional.  We will see that the right
adjoint in our Quillen equivalence always admits a product decomposition, but in general the left adjoint does not
admit a corresponding coproduct formula.

\begin{rem}
By applying the Lower Triangular Axiom to the ``identity matrix," it is immediate that Proposition~\ref{prod_weq} holds
in any semi-stable model category.
\end{rem}

\begin{prop}
Every complete and cocomplete abelian category admits the structure of a semi-stable model category.
\end{prop}

\begin{proof}
Suppose that $\mcC$ is abelian with all limits and colimits.  It is well-known that declaring the weak equivalences to be the isomorphisms gives a model
structure on $\mcC$ (and here, every map is both a fibration and a cofibration).  In this case, all ``homotopy adjectives"
become vacuous.  It is immediate that all parts of Convention~\ref{conv} are satisfied.  Moreover, Pushout-Product and
Colimit-Product Coherence follow from the fact that the natural map $X \oplus Y \rightarrow X \times Y$ is always an isomorphism.

For the Lower Triangular Axiom, suppose that $f=\begin{pmatrix} a & 0 \\ c & d \end{pmatrix}$ is such a map with $a$ and
$d$ isomorphisms.  We can write down the inverse of $f$ explicitly as
\[
f^{-1} = \begin{pmatrix} a^{-1} & 0 \\ -d^{-1}ca^{-1} & d^{-1} \end{pmatrix}
\]
as one can check.  Hence the Lower Triangular Axiom holds and $\mcC$ is semi-stable.
\end{proof}

Many stable model categories also satisfy the axioms for a semi-stability.  Recall that a pointed model category is
\textit{stable} if its homotopy category is triangulated under its natural loop and suspension.
For a definition that avoids reference to the homotopy category, Mark Hovey points out in \cite[Chapter 7]{MH} that this
is equivalent to the coincidence of homotopy pullback and pushout squares in the underlying model category.  Hence
Pushout-Product Coherence is to be expected as products and pullbacks interact nicely.  The third axiom will
follow whenever products of weak equivalences are weak equivalences and so are the maps $X \vee Y \rightarrow X \times Y$
for $X$ and $Y$ are cofibrant.

\begin{prop}\label{lower_triang_prop}
The Lower Triangular Axiom holds in any stable model category in which all objects are fibrant.
\end{prop}

\begin{proof}
Take a map $f = \begin{pmatrix} f_{11} & 0 \\ f_{21} & f_{22} \end{pmatrix}$ as in the statement of the axiom.  First one
notes that the diagram
\[
\begindc{0}[3]
\obj(-22,9)[1]{$A_1$}
\obj(0,9)[2]{$A_1 \vee A_2$}
\obj(22,9)[3]{$A_2$}
\obj(-22,-9)[4]{$A_1$}
\obj(0,-9)[5]{$A_1\times A_2$}
\obj(22,-9)[6]{$A_2$}
\mor{1}{2}{$i_1$}
\mor{2}{3}{$0 \vee 1$}
\mor{4}{5}{$1 \times 0$}
\mor{5}{6}{$p_2$}
\mor{1}{4}{$f_{11}$}[\atright, \solidarrow]
\mor{2}{5}{$f$}
\mor{3}{6}{$f_{22}$}
\enddc
\]
commutes precisely because $f_{12} = 0$.  As $A_1$ and $A_2$ are cofibrant, the top row is a homotopy cofiber sequence
in the homotopy category.  Likewise, since all objects are fibrant, the bottom row is a homotopy fiber sequence.
Of course in the homotopy category, fiber and cofiber sequences coincide.  Since $f_{11}$ and $f_{22}$ are weak
equivalences, we now have a map of fiber sequences that is an isomorphism on the fiber and the base.  Thus $f$ represents
an isomorphism in the homotopy category, and is therefore a weak equivalence.
\end{proof}

\begin{example}
The simplest example of such a stable model category is the category of chain complexes of $R$-modules.  Here the
weak equivalences are the homology isomorphisms and the fibrations are the surjections.  Of course this category
is also abelian and thus may be regarded as semi-stable with the isomorphisms as the weak equivalences.  Our main theorem
says something very different in each of these cases.
\end{example}


\section{Indexing Categories}\label{index_cat_sect}

Categories with a certain kind of rigid structure will be instrumental throughout the development of our main result.
The tight structural properties of these categories will give us control over the indexing of crucial (co)product decompositions.
In essence, these \emph{indexing categories} serve to parameterize images and cokernels.

For this, we will need to recall some basic facts about EI-categories.  First, in an EI-category all retracts are isomorphisms.
In any category with pullbacks, the condition that all retracts are isomorphisms is equivalent to every map being monic.
Hence all maps are monic in an EI-category with pullbacks.  An additional finiteness assumption is all that we require in
the following.

\begin{defn}
A small EI-category $\mcI$ with pullbacks is an \textit{indexing category} if each comma category $\mcI \downarrow a$
has a skeleton with a finite number of objects.
\end{defn}

The last condition simply asserts that for each object $a$ there are only finitely many maps to $a$ up
to ``covering" equivalence.  Moreover, such covering equivalences are unique since all maps in an indexing
category are monic.  We will denote a skeleton of $\mcI \downarrow a$ by $\sk (\mcI \downarrow a)$.  One important
structural aspect of indexing categories is given in the next proposition.

\begin{prop}\label{poset}
Let $\mcI$ be an indexing category.  For each object $a$, the category $\sk (\mcI \downarrow a)$ has the structure of a
finite partially ordered set.
\end{prop}

\begin{proof}
This is even true before taking skeleta (except for the finiteness).  The ordering $\leq$ is defined as follows:  for maps $i: x \rightarrow a$
and $j: y \rightarrow a$, we declare $i \leq j$ if there is a commutative triangle
\[
\begindc{0}[3]
\obj(-10,8)[x]{$x$}
\obj(10,8)[y]{$y$}
\obj(0,-8)[a]{$a$}
\mor{x}{y}{$k$}
\mor{x}{a}{$i$}[\atright, \solidarrow]
\mor{y}{a}{$j$}
\obj(3,-9){.}
\enddc
\]
Checking that this gives a partial ordering is routine.
\end{proof}


\begin{example}\label{poset_ex1}
Let $P$ be a partially ordered set with unique greatest lower bounds.  Suppose that the segment $\{x \in P \ | \ x \leq a\}$
is finite for each $a \in P$.  Then the category formed from the poset $P$ (in the usual way) is an indexing category.  Hence
any finite tree forms an indexing category, as does the subgroup lattice of a finite group.
\end{example}

\begin{example}\label{N_subsets}
Given a subset $A$ of the natural numbers (without 0), let $A_+$ denote $A \cup \{0\}$, where $0$ will always play the
role of the basepoint.  Let $\mcI$ denote the category with objects the finite sets $A_+$ and morphisms the based injective
functions.  As pullbacks here are given by intersections, it is easy to see that $\mcI$ is an indexing category.
\end{example}

\begin{example}\label{Aut_G}
(After example 11.2 of \cite{tD}.)  Fix a finite group $G$ and a homogeneous $G$-set $G/H$.  Denote by $\Gamma(G/H)$ the category with $G/H$ as its only
object and equivariant $G$-maps as the morphisms.  It is well known that all equivariant maps $G/H \rightarrow G/H$ are
automorphisms, so $\Gamma(G/H)$ is an EI-category with pullbacks.  The finiteness condition is obviously met, so $\Gamma(G/H)$
is an indexing category.  Functors from this category into the category of $R$-modules give left $R\mathrm{Aut}(G/H)$-modules.
\end{example}

\begin{example}\label{mono_G}
Let $G$ be a group and let $\mcM_G$ be the category of finite $G$-sets and equivariant monomorphisms.  Pullbacks correspond
to intersections and the finiteness condition on the skeleta is clearly satisfied.  Thus $\mcM_G$ is an indexing category.
\end{example}

\begin{example}\label{ordered_maps}
Call a $\Gamma$-map $i: \mplus \rightarrow \nplus$ \textit{ordered} if $i(x) < i(y)$ whenever $x < y$.  Letting $\mcO$ denote
the subcategory of $\Gamma$ consisting of the ordered maps, we see that $\mcO$ is an indexing category.  It is clear that
$\mcO(\mplus, \nplus)$ is in one-to-one correspondence with the subsets of $\bfn = \{1, 2, \ldots, n\}$ of order $m$.
Under this correspondence, the pullback of two maps in $\mcO$ corresponds to the intersection of the subsets they represent.
Moreover, we have $\sk (\mcO \downarrow \nplus) = \mcO \downarrow \nplus$ since $\mcO$ has no non-identity isomorphisms.
\end{example}

This last example has some additional structure that should be emphasized.
Every ordered map $i: \mplus \rightarrow \nplus$ has a natural dual $i^*: \nplus \rightarrow \mplus$ that collapses the
complement of the image of $i$ to the basepoint and satisfies $i^* \circ i = 1$.  These two properties in fact characterize the ``collapse"
map $i^*$.  All such collapse maps give a subcategory of $\Gamma$.

It is obvious that $(k \circ i)^* = i^* \circ k^*$, so that the category $\mcO^*$ of collapse maps is isomorphic to
$\mcO^\op$.  Moreover, a commutative square
\[
\begindc{0}[3]
\obj(-9,8)[a]{$\mplus$} \obj(9,8)[b]{$\nplus$} \obj(-9,-8)[c]{$\rplus$} \obj(9,-8)[d]{$\splus$} \mor{a}{b}{$i$} \mor{c}{d}{$l$}
\mor{b}{d}{$k$} \mor{a}{c}{$j$}[\atright, \solidarrow]
\enddc
\]
in $\mcO$ is a pullback if and only if $j \circ i^* = l^* \circ k$ (recall that pullbacks are intersections here).  This ``interchange
law" implies
that all composites of the form $j \circ i^*$ ($i, j \in \mcO$) yield a subcategory of $\Gamma$:  in order to compose
two such maps, one must use the appropriate pullback to swap the two middle terms.  Note that this new parent category contains
both $\mcO$ and $\mcO^*$ as subcategories.  In the next section we will see that every indexing category allows
for a construction of this sort; see Example~\ref{iso_ex} below.


\section{Conjugate Pairs of Small Categories}\label{conj_pair_section}

\subsection{Categories admitting conjugation}

Let $\mcU$ be a category with subcategories $\mcP$ and $\mcQ$, all three having the same objects.  We say that
$\mcU$ \textit{factors as} $\mcQ \circ \mcP$ if every morphism of $\mcU$ is expressible
as a composition $q \circ p$ for some maps $q \in \mcQ$ and $p \in \mcP$.  In this case we shall write $\mcU = \mcQ
\circ \mcP$.  For us, such factorizations will not necessarily be unique on the nose, but only up to the correct notion of equivalence
of maps in an indexing category (see the first axiom below).  We will write $\Iso(\mcC)$ for the class of isomorphisms in a category $\mcC$.

\begin{defn}\label{reg_fact_def}
Suppose that $\mcU$ is a small category which factors as $\mcU = \mcI \circ \mcA$, where $\mcI$ is an indexing
category with $\Iso(\mcI) \subseteq \Iso(\mcA)$.  We say that this factorization \emph{admits conjugation} if the
following two axioms hold:
\begin{itemize}
\item  The factorization $\mcU = \mcI \circ \mcA$ is unique up to lifting isomorphisms in $\mcI$.  Precisely,
for each commutative square
\[
\begindc{0}[3]
\obj(-8,8)[1]{$a$} \obj(8,8)[2]{$b$} \obj(-8,-8)[3]{$c$}
\obj(8,-8)[4]{$d$} \mor{1}{2}{$\alpha$} \mor{3}{4}{$j$}
\mor{1}{3}{$\beta$}[\atright, \solidarrow]
\mor{2}{4}{$i$} 
\mor{3}{2}{$ $}[\atleft, \dasharrow]
\enddc
\]
with $\alpha, \beta \in \mcA$ and $i, j \in \mcI$, the indicated lift exists and is an isomorphism in $\mcI$.
\item  Given $\alpha: a \rightarrow b$ in $\mcA$ and $i: c \rightarrow b$ in $\mcI$, there is a pullback in $\mcU$
of the form
\[
\begindc{0}[3]
\obj(-8,8)[a]{$p$} \obj(8,8)[b]{$a$} \obj(-8,-8)[c]{$c$} \obj(8,-8)[d]{$b$} \mor{a}{b}{$i'$} \mor{c}{d}{$i$}
\mor{b}{d}{$\alpha$} \mor{a}{c}{$\alpha'$}[\atright, \solidarrow]
\enddc
\]
with $\alpha' \in \mcA$ and $i' \in \mcI$.  Moreover, the natural isomorphism relating two such pullback squares is given
by an isomorphism in $\mcI$.
\end{itemize}
\end{defn}

In essence, these axioms serve to generalize the notion of image and inverse image, ensuring each is unique up to the
correct kind of isomorphism.  The motivation for this definition comes from the following example.

\begin{example}\label{unbased_admits_cong}
Let us say that a map $\gamma$ in $\Gamma$ is \textit{regular} if
$\gamma^{-1}(0) = \{0\}$. That is, $\gamma$ is regular if it sends only the basepoint to the basepoint.  Let
$\mcU$ denote the subcategory of regular maps.  By forgetting the basepoints, we see that $\mcU$ is equivalent to the
category of finite (unbased) sets.

Take as our indexing category the category $\mcO$ of ordered maps in $\Gamma$ (see Example~\ref{ordered_maps}).  By adding
disjoint basepoints, we will consider the category $\mcE$ of unbased epimorphisms as a subcategory of $\Gamma$.  Every map
in $\mcU$ clearly factors as an epimorphism followed by the inclusion of the image into the original codomain.  Once this
image subset is uniquely represented by an $\mcO$-map, we see that $\mcU$ factors as $\mcO \circ \mcE$.  It is immediate that the first axiom
holds since such factorizations are unique on the nose (note that $\mcO$ has no non-identity isomorphisms).

For the second axiom, it is instructive to check
that the pullback of an $\mcE$-map $\alpha: \mplus \rightarrow \nplus$ along an $\mcO$-map $i: \kplus \rightarrow \nplus$
is the inverse image under $\alpha$ of the subset of $\nplus$ represented by $i$.  It is then clear that parallel partners
in the pullback square belong to the same subcategory.  Again, $\mcO$ is sparse enough that the uniqueness requirement
is trivially met, so the second axiom holds.  Hence the factorization $\mcU = \mcO \circ \mcE$ admits conjugation.
\end{example}

We will see that whenever a factorization $\mcU = \mcI \circ \mcA$ admits conjugation, one can construct a new
category $\mcB$, obtained from $\mcU$ by attaching formal pre-compositions by maps in $\mcI^\op$.  In this way
$\mcB$ itself factors as $\mcB = \mcI \circ \mcA \circ \mcI^\op$, giving the three-fold factorizations alluded
to in the introduction.  


\subsection{Conjugate pairs}\label{conj_def_ss}

Throughout, suppose that $\mcU = \mcI \circ \mcA$ is a factorization admitting conjugation.  We will describe how
to construct a new category $\mcB$ containing all of the original data as subcategories, with $\mcB$ factoring
as $\mcB = \mcI \circ \mcA \circ \mcI^\op$.  We will refer to the pair of categories $(\mcB, \mcA)$ as a \emph{conjugate
pair}.  This construction is a generalization of the
\emph{induction categories} of \cite{tD}, or the category $\omega(G)$ of \cite{TW}, both well-known to representation theorists.  In fact, if one takes
$\mcA = \Iso(\mcI)$ then our two axioms for conjugation are trivially satisfied, and the construction we give reduces
to the others.

The category $\mcB$ will have the same objects as $\mcU$.   A morphism $\beta: a \rightarrow b$ in $\mcB$ will be represented
by a diagram
\[
a \stackrel{i}{\longleftarrow} a_1 \stackrel{\alpha}{\longrightarrow} b_1 \stackrel{j}{\longrightarrow} b
\]
where $i, j \in \mcI$ and $\alpha \in \mcA$.  Writing $i^*: a \rightarrow a_1$ for the formal opposite of $i$, we
shall write $\beta = j \circ \alpha \circ i^*$.  In order to get an honest category, we will have to identify some of these morphisms.  We will declare the morphism above
to be equivalent to
\[
a \stackrel{k}{\longleftarrow} a_2 \stackrel{\gamma}{\longrightarrow} b_2 \stackrel{l}{\longrightarrow} b
\]
if there exist isomorphisms $\varphi, \psi \in \mcI$ making the entire diagram
\[
\begindc{0}[3]
\obj(-8,8)[1]{$a_1$}
\obj(8,8)[2]{$b_1$}
\obj(-8,-8)[3]{$a_2$}
\obj(8,-8)[4]{$b_2$}
\obj(-22,0)[5]{$a$}
\obj(22,0)[6]{$b$}
\mor{1}{2}{$\alpha$}
\mor{3}{4}{$\gamma$}
\mor{1}{3}{$\varphi$}[\atright, \solidarrow]
\mor{2}{4}{$\psi$}
\mor{1}{5}{$i$}[\atright, \solidarrow]
\mor{2}{6}{$j$}
\mor{3}{5}{$k$}
\mor{4}{6}{$l$}[\atright, \solidarrow]
\enddc
\]
commute.

It is easy to check that this gives an equivalence relation, and one can take the morphisms in $\mcB$ to be
equivalence classes of such diagrams.  Alternatively, we may take the morphisms of $\mcB$ to be all formal
composites of the form $\beta = j \circ \alpha \circ i^*$ as above, with the understanding that such representations are
not unique.  It is convenient to think of $\beta$ as admitting a three-fold
factorization
\[
\begindc{0}[3]
\obj(-8,8)[1]{$a$} \obj(8,8)[2]{$b$} \obj(-8,-8)[3]{$a_1$}
\obj(8,-8)[4]{$b_1$} \mor{1}{2}{$\beta$} \mor{3}{4}{$\alpha$}
\mor{1}{3}{$i^*$}[\atright, \solidarrow]
\mor{4}{2}{$j$}[\atright, \solidarrow] 
\enddc
\]
with such factorizations unique only up to adjustments by isomorphisms in $\mcI$.  The latter point of view leads to simpler
notation, so this is the approach we will take.  We will refer to $i^*$ as the \emph{cokernel} of $\beta$, and likewise we
will call $j$ its \emph{image}.

\begin{rem}
In any two three-fold factorizations of the given map $\beta$, the cokernels are equivalent morphisms
in the comma category $\mcI \downarrow a$.  Likewise, the images are equivalent in $\mcI \downarrow b$.  Hence
three-fold factorizations are unique if we require the $\mcI$-components to lie in a fixed skeleton of the relevant
comma category.
\end{rem}

Composition of such morphisms is defined in terms of pullbacks in the indexing category $\mcI$ and the two axioms
for conjugation.  The composition of $a \stackrel{i}{\longleftarrow} w \stackrel{\alpha}{\longrightarrow} x
\stackrel{j}{\longrightarrow} b$ with $b \stackrel{k}{\longleftarrow} y \stackrel{\gamma}{\longrightarrow} z
\stackrel{l}{\longrightarrow} c$ is displayed in the following diagram:
\[
\begindc{0}[3]
\obj(-22,19)[1]{$p$}
\obj(0,19)[2]{$q$}
\obj(22,19)[3]{$r$}
\obj(-33,0)[4]{$w$}
\obj(33,0)[5]{$z$}
\obj(-44,-19)[6]{$a$}
\obj(44,-19)[7]{$c$} \obj(46,-20)[p]{.}
\obj(-11,0)[8]{$x$}
\obj(11,0)[9]{$y$}
\obj(0,-19)[10]{$b$}
\mor{1}{2}{$\alpha'$}
\mor{2}{3}{$\gamma'$}
\mor{1}{4}{$k''$}[\atright, \solidarrow]
\mor{2}{8}{$k'$}[\atright, \solidarrow]
\mor{2}{9}{$j'$}
\mor{3}{5}{$j''$}
\mor{4}{8}{$\alpha$}
\mor{9}{5}{$\gamma$}
\mor{4}{6}{$i$}[\atright, \solidarrow]
\mor{8}{10}{$j$}
\mor{9}{10}{$k$}[\atright, \solidarrow]
\mor{5}{7}{$l$}
\enddc
\]
The middle diamond is the pullback of $j$ along $k$, formed in $\mcI$.  The upper-left square is the pullback
of $\alpha$ along $k'$ per the second axiom for conjugation; hence $\alpha' \in \mcA$ and $k'' \in \mcI$.  According
to the first axiom, the map $\gamma \circ j'$ admits a factorization of the form $j'' \circ \gamma'$ where
$j'' \in \mcI$ and $\gamma' \in \mcA$; this is the upper-right square.  Hence the composition is given by
$(l j'') \circ (\gamma' \alpha') \circ (ik'')^*$.

Of course, none of the steps in the composition are necessarily uniquely determined.  However, it is easy to check that
different choices would lead to equivalent morphisms (thanks to the axioms for conjugation), so the composition is in fact well-defined.
This completes the description of the category $\mcB = \mcI \circ \mcA \circ \mcI^\op$.

\begin{defn}\label{conj_pair_def}
A pair $(\mcB, \mcA)$ of small categories is a \textit{conjugate pair} if there exists a factorization
$\mcU = \mcI \circ \mcA$ admitting conjugation with $\mcB$ equivalent to the category $\mcI \circ \mcA \circ \mcI^\op$.
\end{defn}

\begin{prop}\label{composition_relations}
Suppose that $(\mcB, \mcA)$ is a conjugate pair arising from the factorization $\mcU = \mcI \circ \mcA$.
\begin{enumerate}
\item[(a)]  For maps $i, j \in \mcI$, we have $(j \circ i)^* = i^* \circ j^*$ whenever the composition is defined.
\item[(b)]  For any map $i \in \mcI$, we have $i^* \circ i = 1$.
\item[(c)]  If $i \in \mcI$ and both $\alpha, i^* \circ \alpha \in \mcA$, then $i$ must be an isomorphism.
\end{enumerate}
\end{prop}

\begin{proof}
The first two assertions follow immediately from the law of composition in $\mcB$.  For the second, one needs only
to recall that all maps in $\mcI$ are monic, hence the pullback (formed in $\mcI$) of
$i$ along itself may be given by completing the square with identity maps.

For the third claim, suppose that
\[
\begindc{0}[3]
\obj(-8,8)[a]{$p$} \obj(8,8)[b]{$a$} \obj(-8,-8)[c]{$c$} \obj(8,-8)[d]{$b$} \mor{a}{b}{$j$} \mor{c}{d}{$i$}
\mor{b}{d}{$\alpha$} \mor{a}{c}{$\gamma$}[\atright, \solidarrow]
\enddc
\]
is a pullback as in Definition~\ref{reg_fact_def}, so that $i^* \circ \alpha = \gamma \circ j^*$.  Since
$\gamma \circ j^* \in \mcA$, the diagram
\[
\begindc{0}[3]
\obj(-9,8)[1]{$p$}
\obj(9,8)[2]{$c$}
\obj(-9,-8)[3]{$a$}
\obj(9,-8)[4]{$c$} 
\obj(-23,0)[5]{$a$}
\obj(23,0)[6]{$c$}
\mor{1}{2}{$\gamma$}
\mor{3}{4}{$\gamma j^*$}
\mor{1}{3}{$j$}[\atright, \solidarrow]
\mor{2}{4}{$1$}
\mor{1}{5}{$j$}[\atright, \solidarrow]
\mor{2}{6}{$1$}
\mor{3}{5}{$1$}
\mor{4}{6}{$1$}[\atright, \solidarrow]
\enddc
\]
shows that $j$ must be an isomorphism in $\mcI$, hence $\alpha = i \circ \gamma \circ j^{-1}$.  Since
$\Iso(\mcI) \subseteq \Iso(\mcA)$, we see that $\gamma \circ j^{-1} \in \mcA$, and the same sort of equivalence
diagram shows that $i$ is an isomorphism.
\end{proof}


\subsection{Examples}\label{examples_section}

We start with the smallest and perhaps most instructive example.

\begin{example}\label{idem_example}
Let $\mcI$ be the category consisting of two objects and only one non-identity map, say
\[
\begindc{0}[3]
\obj(-7,0)[0]{$0$} \obj(7,0)[1]{$1$} \mor{0}{1}{$i$}
\obj(9,-1){.}
\enddc
\]
Taking $\mcA$ to be discrete, the resulting category $\mcB = \mcI \circ \mcI^\op$ has diagrammatic representation
\[
\begindc{0}[3]
\obj(-7,0)[0]{$0$} \obj(7,0)[1]{$1$}
\obj(-7,1)[0']{} \obj(7,1)[1']{}
\obj(-7,-1)[0'']{} \obj(7,-1)[1'']{}
\mor{0'}{1'}{$i$} \mor{1''}{0''}{$i^*$}
\enddc
\]
where $i^* \circ i = 1$.  Note then that $i \circ i^*$ is an idempotent.

At this point it is instructive to recall what our main theorem would say here.  It would assert that, for semi-stable
model categories $\mcC$, the functor category $[\mcB^\op, \mcC]$ is equivalent to $[\mcA^\op, \mcC]$.  As $\mcA$ is
discrete, the latter is simply $\mcC \times \mcC$.  Hence the theorem says
that to give a diagram
\[
\begindc{0}[3]
\obj(-7,0)[0]{$M$} \obj(7,0)[1]{$N$}
\obj(-7,1)[0']{} \obj(7,1)[1']{}
\obj(-7,-1)[0'']{} \obj(7,-1)[1'']{}
\mor{0'}{1'}{$i$} \mor{1''}{0''}{$i^*$}
\enddc
\]
in $\mcC$ with $i \circ i^*$ an idempotent is equivalent to giving two objects in $\mcC$ (namely, $M$ and the ``kernel" of
$i^*$).  Hence our main theorem is essentially a statement about the ability to split idempotents in $\mcC$; this is
the elegant explanation of our work.  (Of course, we are concerned with obtaining a Quillen equivalence and so we
are splitting idempotents in the homotopy category, not $\mcC$ itself.)
\end{example}

\begin{example}\label{poset_ex2}
Let $\mcI$ be the category of a partially ordered set $P$ as in Example \ref{poset_ex1}.  With $\mcA$ discrete, the category
$\mcB= \mcI \circ \mcI^\op$ has the following description.  A morphism $a: x \rightarrow y$ in $\mcB$ is just the statement that
$a \leq x$ and $a \leq y$.  The composition of $a: x \rightarrow y$ and $b: y \rightarrow z$ is the greatest
lower bound of $a$ and $b$.  The category $\mcI$ appears in $\mcB$ as the maps $x: x \rightarrow y$ and similarly the maps
$x: y \rightarrow x$ represent $\mcI^\op$ as a subcategory.  We now see that each map $a : x \rightarrow y$ in $\mcB$
factors uniquely as
\[
\begindc{0}[3]
\obj(-10,8)[x]{$x$}
\obj(10,8)[y]{$y$}
\obj(0,-8)[a]{$a$}
\mor{x}{y}{$a$}
\mor{x}{a}{$a$}[\atright, \solidarrow]
\mor{a}{y}{$a$}[\atright, \solidarrow]
\obj(2,-9){.}
\enddc
\]
Hence $(\mcB, \mcI)$ is a conjugate pair.
\end{example}

\begin{example}\label{iso_ex}
The two previous examples generalize: every indexing category gives rise to a conjugate pair.  Given
an indexing category $\mcI$, let $\mcA$ be the category of isomorphisms in $\mcI$.  It is immediate that the axioms
for conjugation are satisfied, and so we obtain a category $\mcB$ with $(\mcB, \Iso(\mcI))$ a conjugate pair.  This
is the non-additive version of the induction categories of \cite{tD}.
\end{example}


\begin{example}\label{P_example}
Here we obtain Pirashvili's example \cite{TP}, namely that $(\Gamma, \mcE)$ is a conjugate pair.  Recall the terminology
and notation of Example~\ref{unbased_admits_cong}, where we saw that $\mcU = \mcO \circ \mcE$ admits conjugation.  We
claim that the category $\mcB$ is equivalent to $\Gamma$.  Recalling that $\mcO^\op$ is equivalent to the category $\mcO^*$ of collapse maps, we show that every
map in $\Gamma$ admits an internal three-fold factorization of the type $\mcO \circ \mcE \circ \mcO^*$.

Fix a
$\Gamma$-map $\gamma: \mplus \rightarrow \nplus$ and let $i: \rplus \rightarrow \mplus$ be the $\mcO$-map representing the
complement of the ``kernel" $\gamma^{-1}(0)$ of $\gamma$.  Upon using $i^*$ to collapse the kernel to a point, $\gamma$ induces a regular map $\overline{\gamma}: \rplus \rightarrow
\nplus$.  This regular map $\overline{\gamma}$ then admits a factorization by an $\mcE$-map $\gamma': \rplus \rightarrow \splus$ followed by the map $j: \splus
\rightarrow \nplus$ representing the image of $\gamma$ (which is also the image of $\overline{\gamma}$).  Thus $\gamma$ admits a three-fold factorization
as $\gamma = j \circ \gamma' \circ i^*$:
\[
\begindc{0}[3]
\obj(-9,8)[1]{$\mplus$} \obj(9,8)[2]{$\nplus$} \obj(-9,-8)[3]{$\rplus$}
\obj(9,-8)[4]{$\splus$} \mor{1}{2}{$\gamma$} \mor{3}{4}{$\gamma'$}
\mor{1}{3}{$i^*$}[\atright, \solidarrow]
\mor{4}{2}{$j$}[\atright, \solidarrow] \obj(12,-9){.}
\enddc
\]
Hence $(\Gamma, \mcE)$ is a conjugate pair.  By assuming all maps are weakly monotone one can also obtain a simplicial version
of this example.
\end{example}

\begin{example}
The previous example may be fattened up a bit.  Let $\mcB$ denote the category with objects the finite based subsets $A_+$
of the natural numbers (with $0$ acting as the basepoint) and morphisms the based maps.  With $\mcA$ the subcategory
of regular based surjections and $\mcI$ as in Example~\ref{N_subsets}, we have that $(\mcB, \mcA)$ forms a conjugate pair.
\end{example}

\begin{example}
Let $\mcB$ denote the category of $\Gamma$-maps $\beta$ such that the inverse image $\beta^{-1}(x)$ of each nonzero point $x$ is either empty or
a singleton.  That is, $\beta \in \mcB$ may send lots of elements to the basepoint, but modulo this, it is injective.
With $\Sigma$ denoting the category of regular permutations, $\mcB$ factors as $\mcO \circ \Sigma \circ \mcO^*$
and $(\mcB, \Sigma)$ is a conjugate pair.
\end{example}


\section{Two Natural Decompositions}\label{nat_section}

For the remainder of this paper, $(\mcB, \mcA)$ will denote a fixed conjugate pair arising from a factorization
$\mcU = \mcI \circ \mcA$.  In this section we construct the bimodule $\mcU_+: \mcA^\op \times \mcB \rightarrow \Sets_*$ which will induce
our Quillen equivalence.  In Propositions~\ref{reg_free} and \ref{reg_gen} we show that this bimodule
is ``free" as a right $\mcA$-module and a ``generator" in its left $\mcB$-module structure.

\begin{defn}
The maps in $\mcB$ lying in the subcategory $\mcU = \mcI \circ \mcA$ will be called the \textit{regular maps}.  If a map is not regular,
we will say it is \textit{singular}.
\end{defn}

We remark that a map $\beta = j \circ \alpha \circ i^*$ is regular if and only if $i$ is an isomorphism in $\mcI$.  Since
$\Iso(\mcI) \subseteq \Iso(\mcA)$, such a map is always equal to one with the identity map as its cokernel, as shown
by the diagram
\[
\begindc{0}[3]
\obj(-9,8)[1]{$a_1$}
\obj(9,8)[2]{$b_1$}
\obj(-9,-8)[3]{$a$}
\obj(9,-8)[4]{$b_2$} \obj(12,-10)[p]{.}
\obj(-23,0)[5]{$a$}
\obj(23,0)[6]{$b$}
\mor{1}{2}{$\alpha$}
\mor{3}{4}{$\alpha i^{-1}$}
\mor{1}{3}{$i$}[\atright, \solidarrow]
\mor{2}{4}{$1$}
\mor{1}{5}{$i$}[\atright, \solidarrow]
\mor{2}{6}{$j$}
\mor{3}{5}{$1$}
\mor{4}{6}{$j$}[\atright, \solidarrow]
\enddc
\]

\begin{bold_notation}
Let $\mcS(a,b)$ denote the set of singular maps in $\mcB(a, b)$.  We let $\mcU(a,b)_+$ denote the quotient set
\[
\mcU(a,b)_+ = \mcB(a,b) / \mcS(a,b)
\]
where we take the singular maps as a basepoint.  As $\mcB(a,b)$ is the
disjoint union of the regular and singular maps, this may also be regarded as the set of regular maps together
with a disjoint basepoint (thus this notation is sensible).
\end{bold_notation}

\begin{prop}\label{closure}
In any conjugate pair $(\mcB, \mcA)$, the singular maps have the following ``ideal-like" properties:
\begin{enumerate}
\item[(a)]  If $\alpha \in \mcA$ and $\gamma$ is singular, then $\gamma \circ \alpha$ is singular when defined.
\item[(b)]  If $\beta \in \mcB$ and $\sigma$ is singular, then $\beta \circ \sigma$ is singular when defined.
\end{enumerate}
\end{prop}

\begin{proof}
We prove only the first statement; the proof of the second is similar (and simpler).  Factor $\gamma$ as $\gamma = j \circ
\gamma' \circ i^*$ and suppose that $\gamma \circ \alpha$ is regular.  If $i^* \circ \alpha$ factors as
$i^* \circ \alpha = l \circ \delta \circ k^*$ then the cokernel of $\gamma \circ \alpha$ is still $k^*$.  Since
$\gamma \circ \alpha$ is regular, $k$ must be an isomorphism in $\mcI$, and we can now arrange to take $k$ to be the identity by
our previous remarks.  From $i^* \circ \alpha = l \circ \delta$ we compose on the left with $l^*$, and Proposition~\ref{composition_relations}
gives that $(i \circ l)^* \circ \alpha = \delta \in \mcA$, and hence $i \circ l$ must be an isomorphism in $\mcI$.  We now
see that $i$ is a retract in the EI-category $\mcI$, and hence $i$ is an isomorphism.  Thus $\gamma$ is regular.
\end{proof}

\begin{prop}
Given a conjugate pair $(\mcB, \mcA)$, the construction $\mcU(-,-)_+$ defines a functor $\mcA^\op \times
\mcB \rightarrow \Sets_*$.
\end{prop}

\begin{proof}
Suppose we are given morphisms $\alpha: a \rightarrow b$ in $\mcA$
and $\beta: c \rightarrow d$ in $\mcB$.  We then get an associated
map $\mcB(b,c) \rightarrow \mcB(a,d)$ which sends a map $\gamma: b
\rightarrow c$ to the composite $\beta \circ \gamma \circ \alpha$.
By Proposition~\ref{closure}, this sends singular maps to
singular maps.  Hence this passes down to quotients, as desired.
\end{proof}

The functor $\mcU_+: \mcA^\op \times \mcB \rightarrow \Sets_*$ is the bimodule desired for our Morita equivalence.
We will refer to this functor as the \textit{regular bimodule}.

\begin{lemma}\label{A_funct}
Suppose that  $(\mcB, \mcA)$ is a conjugate pair and fix an object $a$ of $\mcA$.  Every map $\beta:
b \rightarrow c$ in $\mcB$ induces a map
\[
\beta_*: \bigvee_{i \in \sk(\mcI \downarrow b)} \mcA(a, \dom(i))_+ \longrightarrow
\bigvee_{j \in \sk(\mcI \downarrow c)} \mcA(a, \dom(j))_+ \ .
\]
Furthermore, this assignment is functorial, so that wedge sums of
the form
\[
\bigvee_{i \in \sk(\mcI \downarrow b)} \mcA(a, \dom(i))_+
\]
give a functor $\mcB \rightarrow \Sets_*$ in the $b$-variable.
\end{lemma}

Verifying that this defines a functor is fairly straightforward,
and at worst consists of checking a few special cases.   The proof uses only three-fold factorizations
and Proposition~\ref{closure}, therefore we will only describe the induced map $\beta_*$.

Suppose we are in the summand corresponding to $i: b' \rightarrow
b$ in $\mcI$, and let $\alpha: a \rightarrow b'$ be a map in
$\mcA$.  We consider the composite
\[
a \stackrel{\alpha}{\longrightarrow} b'
\stackrel{i}{\longrightarrow} b \stackrel{\beta}{\longrightarrow}
c.
\]
If this composite is singular, we define $\beta_*(\alpha)$ to be
the basepoint.  Otherwise it is regular, and thus admits a
factorization
\[
\begindc{0}[3]
\obj(-8,8)[1]{$a$} \obj(8,8)[2]{$c$} \obj(-8,-8)[3]{$a$}
\obj(8,-8)[4]{$c'$} \mor{1}{2}{$\beta  i  \alpha$}
\mor{3}{4}{$\alpha'$} \mor{1}{3}{$1^*$}[\atright,
\solidarrow] \mor{4}{2}{$j$}[\atright, \solidarrow]
\enddc
\]
where $\alpha' \in \mcA$ and $j \in \sk(\mcI \downarrow c)$.  Note that when factored in this form, $\alpha'$ is uniquely
determined.  We then let $\beta_*(\alpha) = \alpha'$, corresponding to the summand indexed by $j$.

The next result states
that the bimodule $\mcU_+: \mcA^\op \times \mcB \rightarrow
\Sets_*$ is free as a right $\mcA$-module.  Note that the naturality claim in the following is well-posed by
the lemma.

\begin{prop}\label{reg_free}
Suppose that $(\mcB, \mcA)$ is a conjugate pair.  For each pair of objects $a$ and $b$ of $\mcB$, there is an isomorphism of based sets
\[
\mcU(a,b)_+ = \bigvee_{i \in \sk(\mcI \downarrow b)} \mcA(a, \dom(i))_+
\]
which is natural in both variables.  In other words, for each
object $b$ of $\mcB$ there is a natural equivalence
\[
\mcU(-,b)_+ = \bigvee_{i \in \sk(\mcI \downarrow b)} \mcA(-,\dom(i))_+
\]
of functors $\mcA^\op \rightarrow \Sets_*$ and these equivalences
vary naturally with $b$.
\end{prop}

The idea of the proof is simple:  every regular map $\gamma: a \rightarrow b$ admits a factorization by a map
$\alpha: a \rightarrow b'$ in $\mcA$ followed by the inclusion $i: b' \rightarrow b$ of the image of $\gamma$ back
into the codomain.  Once a skeleton has been fixed, such a factorization is unique.  Hence under our isomorphism,
$\gamma$ corresponds to $\alpha \in \mcA(a,b')_+$, landing in the summand indexed by $i$.

Next we carry out a similar analysis in the other variable.

\begin{lemma}\label{reg_funct}
Suppose that $(\mcB, \mcA)$ is a conjugate pair and fix an object $c$ of $\mcB$.  Every map $\beta:
a \rightarrow b$ in $\mcB$ induces a map
\[
\beta^*: \bigvee_{i \in \sk(\mcI \downarrow b)} \mcU(\dom(i),c)_+ \longrightarrow
\bigvee_{j \in \sk(\mcI \downarrow a)} \mcU(\dom(j),c)_+ \ .
\]
Furthermore, this assignment is functorial, so that wedge sums of
the form
\[
\bigvee_{i \in \sk(\mcI \downarrow b)} \mcU(\dom(i),c)_+
\]
give a functor $\mcB^\op \rightarrow \Sets_*$ in the $b$-variable.
\end{lemma}

\begin{prop}\label{reg_gen}
Suppose that $(\mcB, \mcA)$ is a conjugate pair.  For each pair of objects $b$ and $c$ of $\mcB$, there is an isomorphism of based sets
\[
\mcB(b,c)_+ = \bigvee_{i \in \sk(\mcI \downarrow b)} \mcU(\dom(i),c)_+
\]
which is natural in the first variable.
\end{prop}

We remark that in general the isomorphisms of Proposition~\ref{reg_gen} are not natural in the second variable.  Furthermore, by the
finiteness condition on the indexing category $\mcI$, the wedge sums of Propositions~\ref{reg_free} and \ref{reg_gen}
consist of only a finite number of summands.

As before, verifying the assorted claims is rather formal (yet very tedious) once the definition of the induced map
$\beta^*$ is made clear.  To that end, suppose that $\beta: a \rightarrow b$ has three-fold factorization
\[
\begindc{0}[3]
\obj(-8,8)[1]{$a$} \obj(8,8)[2]{$b$} \obj(-8,-8)[3]{$a_1$}
\obj(8,-8)[4]{$b_1$} \mor{1}{2}{$\beta$} \mor{3}{4}{$\beta_1$}
\mor{1}{3}{$i_1^*$}[\atright, \solidarrow]
\mor{4}{2}{$j_1$}[\atright, \solidarrow] 
\enddc
\]
chosen with respect to the skeleta.  We shall describe the map
\[
\beta^*: \bigvee_{i \in \sk(\mcI \downarrow b)} \mcU(\dom(i),c)_+ \longrightarrow \bigvee_{j \in \sk(\mcI \downarrow a)}
\mcU(\dom(j),c)_+
\]
on the summand corresponding to a fixed map $i: b' \rightarrow b$
in $\mcI$.  (This map will always send regular maps to regular
maps, so the basepoint is of no concern here.)  The composite
\[
a_1 \stackrel{\beta_1}{\longrightarrow} b_1
\stackrel{j_1}{\longrightarrow} b \stackrel{i^*}{\longrightarrow} b'
\]
admits a factorization as a map $i_2^*: a_1 \rightarrow a_2$ followed by a regular map $\beta_2: a_2 \rightarrow b'$ (so
$\beta_2$ is simply the last two legs of the three-fold factorization).  Given a regular map $\gamma: b' \rightarrow c$,
we define
\[
\beta^*(\gamma) = \gamma \circ \beta_2
\]
landing in the summand corresponding to $j = i_1 \circ i_2: a_2 \rightarrow a$.


\section{The Induced Adjoint Pair; Free Functors}\label{adjoint_section}

\subsection{First properties}
Thus far, from a conjugate pair $(\mcB, \mcA)$ we obtain an associated bimodule $\mcU_+: \mcA^\op \times \mcB \rightarrow
\Sets_*$.  For a fixed model category $\mcC$, this in turn gives rise to an adjoint pair
\[
\begindc{0}[3]
\obj(-12,0)[1]{$[\mcB^\op, \mcC]$}
\obj(12,0)[2]{$[\mcA^\op, \mcC]$}
\obj(-8,1)[3]{} \obj(8,1)[4]{} \obj(-8,-1)[5]{} \obj(8,-1)[6]{} \mor{3}{4}{$\mcL$} \mor{6}{5}{$\mcR$}
\enddc
\]
between model categories of functors, as described in Section~\ref{Morita_subsection}.  Before we can show that this is a
Quillen equivalence when $\mcC$ is semi-stable, we need some basic properties of this adjoint pair.

\begin{prop}\label{prod_decomp}
Suppose that $\mathcal{C}$ is a model category and that $(\mcB,
\mcA)$ is a conjugate pair of small categories.  For each $G:
\mathcal{A}^{\mathrm{op}} \rightarrow \mathcal{C}$ and $b$ in
$\mcB$, there is an isomorphism
\begin{center}
$\ds \mcR G(b) = \prod_{i \in \sk(\mcI \downarrow b)} G(\dom(i))$
\end{center}
and this is natural in both $G$ and $b$.
\end{prop}

\begin{proof}
Recall that $\mcR G(b) = \Hom^\mcA(\mcU(-,b)_+,G)$.  Now use the natural isomorphism
\[
\mcU(-,b)_+ = \bigvee_{i \in \sk(\mcI \downarrow b)} \mcA(-,\dom(i))_+
\]
supplied by Proposition~\ref{reg_free} and the formal properties of $\Hom^\mcA(-,G)$ to obtain the desired
isomorphism.
\end{proof}

\begin{cor}\label{Quillen_pair}
Suppose that $\mcC$ is a model category and that $(\mcB, \mcA)$ is a conjugate pair.  Then the adjoint pair
$(\mcL, \mcR)$ associated to the regular bimodule is a Quillen pair.
\end{cor}

\begin{proof}
It is enough to check that $\mcR$ preserves fibrations and acyclic fibrations.  This follows immediately from
Proposition~\ref{prod_fib} and our conventions regarding model categories.
\end{proof}

\begin{cor}\label{detects}
Suppose that $\mcC$ is a semi-stable model category and that
$(\mcB, \mcA)$ is a conjugate pair.  Let $\tau: F \rightarrow G$ be
a natural transformation of functors $\mcA^\op \rightarrow \mcC$
with $F$ cofibrant.  Then $\tau$ is a weak equivalence if and only
if $\mcR \tau$ is a weak equivalence.
\end{cor}

\begin{proof}
This follows directly from Proposition~\ref{prod_weq}.
\end{proof}

Corollary~\ref{detects} is one of two major ingredients in the proof of our main theorem.  The missing ingredient is
Proposition~\ref{cofib_equiv}, which states that the unit map $\eta_F: F \rightarrow \mcR(\mcL F)$ is a weak
equivalence whenever $F$ is cofibrant.  In the next section we build the necessary machinery to complete this final step.


\subsection{Free functors and pushouts}\label{free_section}

Our ultimate goal is to establish that the unit map $\eta_F: F \rightarrow \mcR(\mcL F)$ is a weak
equivalence whenever $F$ is cofibrant.  We prove this by an induction argument using the cofibrant generation
hypothesis.  As a first step we examine the free functors.  Throughout this section, $\mcC$ denotes a semi-stable model category.

\begin{lemma}\label{free_decomp}
Suppose that $F: \mcB^\op \rightarrow \mcC$ is the free functor
$F_b^C$.  Then $\mcL F: \mcA^\op \rightarrow \mcC$ is given by
$\mcL F(-) = C \otimes \mcU(-,b)_+$.
\end{lemma}

\begin{proof}
This is just the associativity of our various tensor products,
together with the Yoneda lemma.  By hypothesis, $F(x) = C \otimes
\mcB(x,b)_+$ so that
\begin{eqnarray*}
\mcL F(a) & = & F \otimes_\mcB \mcU(a,-)_+\\
 & = & (C \otimes \mcB(-,b)_+) \otimes_\mcB \mcU(a,-)_+\\
 & = & C \otimes (\mcB(-,b)_+ \otimes_\mcB \mcU(a,-)_+)\\
 & = & C \otimes \mcU(a,b)_+.
\end{eqnarray*}
\end{proof}

\begin{prop}\label{free_equiv}
Suppose that $F: \mcB^\op \rightarrow \mcC$ is a free functor of
the form $F_b^C$ with $C$ a cofibrant object of $\mcC$.  Then the
unit map $\eta_F: F \rightarrow \mcR(\mcL F)$ is a weak
equivalence.
\end{prop}

\begin{proof}
We must show that for each object $a$, the natural map $\eta_F:
F(a) \rightarrow \mcR(\mcL F)(a)$ is a weak equivalence. Recall
that Proposition~\ref{reg_gen} supplies the isomorphism
\[
\mcB(a,b)_+ = \bigvee_{i \in \sk(\mcI \downarrow a)} \mcU(\dom(i),b)_+
\]
which is natural in the first variable.  Using this in conjunction
with Lemma~\ref{free_decomp} we obtain natural isomorphisms
\begin{eqnarray*}
F(a) & = & C \otimes \mcB(a,b)_+\\
 & = & C \otimes \bigvee_{i \in \sk(\mcI \downarrow a)} \mcU(\dom(i),b)_+\\
 & = & \bigvee_{i \in \sk(\mcI \downarrow a)} C \otimes \mcU(\dom(i),b)_+\\
 & = & \bigvee_{i \in \sk(\mcI \downarrow a)} \mcL F(\dom(i)).
\end{eqnarray*}

Now consider the composite
\[\label{not_eta}
\bigvee_{i \in \sk(\mcI \downarrow a)} \mcL F(\dom(i)) \simeq F(a) \stackrel{\eta_F}{\longrightarrow} \mcR(\mcL F)(a)
\simeq \prod_{j \in \sk(\mcI \downarrow a)} \mcL F(\dom(j)). \tag{$\ast$}
\]
Fix maps $i: x \rightarrow a$ and $j: y \rightarrow a$ in the indicated skeleton.
Then $i$ determines the inclusion of the summand $\mcL F(x) = C
\otimes \mcU(x,b)_+$ into the above coproduct, while $j$
determines the projection onto the factor $\mcL F(y) = C \otimes
\mcU(y,b)_+$ out of the product.  Let $f_{ij}$ denote the
corresponding composite through the map ($\ast$) above.  Recall that the
index set $\sk(\mcI \downarrow a)$ is a finite poset.

We claim that $f_{ii}$ is the identity and that $f_{ij} = 0$ when either $i < j$ or when $i$ and $j$ are
incomparable.  Writing $F(a) = C \otimes \mcB(a,b)_+$, one checks that $f_{ij}$ is induced by the composite
\[
\mcU(x,b)_+ \rightarrow \mcB(a,b)_+ \rightarrow \mcU(y,b)_+
\]
which sends a regular map $\gamma: x \rightarrow b$ to the
composite $\gamma \circ i^* \circ j$.  Note that when $i=j$, we
have $i^* \circ j = i^* \circ i=1$, hence $f_{ii}$ is the
identity.  To prove that $f_{ij}$ is the zero map in the other cases it suffices by
Proposition~\ref{closure} to show that $i^* \circ j$ is singular.

In the case that $i < j$ there is a diagram
\[
\begindc{0}[3]
\obj(-10,8)[x]{$x$} \obj(10,8)[y]{$y$} \obj(0,-8)[a]{$a$}
\mor{x}{y}{$k$} \mor{x}{a}{$i$}[\atright, \solidarrow]
\mor{y}{a}{$j$}
\enddc
\]
in $\mcI$ where $k$ is not an isomorphism, and from $i = j \circ k$ we
obtain $i^* \circ j = k^*$.  As $k$ is not an isomorphism, $i^* \circ j$ is singular, hence $f_{ij}$ is the zero map.

The last case we must consider is when $i$ and $j$ are not
comparable.  If the pullback of $i$ and $j$ is
\[
\begindc{0}[3]
\obj(-8,8)[z]{$z$} \obj(8,8)[y]{$y$} \obj(-8,-8)[x]{$x$}
\obj(8,-8)[a]{$a$} \mor{x}{a}{$i$} \mor{z}{y}{$k$} \mor{z}{x}{$l$}[\atright,
\solidarrow] \mor{y}{a}{$j$}
\enddc
\]
then we have $l \circ k^* = i^* \circ j$.  If $k$ is an isomorphism
we obtain $j \leq i$, a contradiction.  Hence $i^* \circ j$ is
singular, and $f_{ij}$ is again the zero map.

We are almost in a position to apply the Lower Triangular Axiom.  We need only turn the finite partially ordered
set $\sk(\mcI \downarrow a)$ into a linearly ordered set, consistent
with the original partial ordering (finiteness is crucial here).
This is achieved by inserting relations $<$ between the
incomparable maps; see, for instance, Theorem~4.5.2 of \cite{Bru}
for details. There is some choice here of course, but this does
not matter.

In such a linear ordering, $i < j$ means one of two things:
either $i < j$ in the original poset, or $i$ and $j$ were not
comparable in the original partial ordering.  In either event,
$f_{ij}$ is the zero map by our previous remarks.  Hence under this linear ordering our map
($\ast$) is lower triangular.  Since $C$ is cofibrant, the Lower Triangular Axiom shows that
($\ast$) is a weak equivalence.  Therefore $\eta_F$ is a weak equivalence, as desired.
\end{proof}

\begin{prop}\label{pushout_equiv}
Suppose that $F_1, F_2: \mcB^\op \rightarrow \mcC$ are free
functors and we have a pushout square
\[
\begindc{0}[3]
\obj(-9,8)[1]{$F_1$} \obj(9,8)[2]{$F_2$} \obj(-9,-8)[3]{$X$}
\obj(9,-8)[4]{$Y$} \mor{1}{2}{} \mor{3}{4}{} \mor{1}{3}{}
\mor{2}{4}{}
\enddc
\]
in which
\begin{enumerate}
\item[(a)]  the map $F_1 \rightarrow F_2$ is a generating cofibration in the category of functors, and
\item[(b)]  the natural map $\eta_X: X \rightarrow \mcR(\mcL X)$ is a weak equivalence.
\end{enumerate}
Then the map $\eta_Y: Y \rightarrow \mcR(\mcL Y)$ is a weak
equivalence.
\end{prop}

\begin{proof}
As $F_1 \rightarrow F_2$ is a cofibration between cofibrant objects, the given pushout square is also a homotopy pushout.
Since $\mcL$ is left Quillen, the same is true of
\[
\begindc{0}[3]
\obj(-9,8)[1]{$\mcL F_1$} \obj(9,8)[2]{$\mcL F_2$} \obj(-9,-8)[3]{$\mcL X$}
\obj(9,-8)[4]{$\mcL Y.$} \mor{1}{2}{} \mor{3}{4}{} \mor{1}{3}{}
\mor{2}{4}{}
\enddc
\]
Applying $\mcR$ to this yields an (objectwise) finite product of homotopy pushouts, so by
Pushout-Product Coherence the square
\[
\begindc{0}[3]
\obj(-11,8)[1]{$\mcR(\mcL F_1)$} \obj(11,8)[2]{$\mcR(\mcL F_2)$} \obj(-11,-8)[3]{$\mcR(\mcL X)$}
\obj(11,-8)[4]{$\mcR(\mcL Y)$} \mor{1}{2}{} \mor{3}{4}{} \mor{1}{3}{}
\mor{2}{4}{}
\enddc
\]
is also a homotopy pushout.  We now have a map
\[
\begindc{0}[3]
\obj(0,0)[1]{$F_1$} \obj(20,0)[2]{$F_2$} \obj(0,-20)[3]{$X$}
\obj(20,-20)[4]{$Y$} \obj(10,10)[5]{$\mcR(\mcL F_1)$}
\obj(30,10)[6]{$\mcR(\mcL F_2)$}
\obj(10,-10)[7]{$\mcR(\mcL X)$}
\obj(30,-10)[8]{$\mcR(\mcL Y)$} \mor{1}{2}{}
\mor{3}{4}{} \mor{1}{3}{} \mor{2}{4}{} \mor{5}{6}{} \mor{7}{8}{}
\mor{5}{7}{} \mor{6}{8}{} \mor{1}{5}{} \mor{2}{6}{} \mor{3}{7}{}
\mor{4}{8}{}
\enddc
\]
of homotopy pushout squares in which the maps $F_i \rightarrow \mcR(\mcL F_i)$ and $X \rightarrow \mcR(\mcL X)$ are
weak equivalences.  It now follows (say, from Proposition~13.5.10 of \cite{PH}) that the map $Y \rightarrow \mcR(\mcL Y)$
is also a weak equivalence.
\end{proof}


\section{The Main Theorem}\label{main_thm_section}

We are now in a position to prove the final missing ingredient.  Again, $\mcC$ denotes a semi-stable model category throughout.

\begin{prop}\label{cofib_equiv}
For an arbitrary cofibrant functor $F: \mcB^\op \rightarrow \mcC$, the unit map $\eta_F: F \rightarrow \mcR(\mcL F)$ is
a weak equivalence.
\end{prop}

\begin{proof}
If $F$ is cofibrant, it is either a cell complex or a retract thereof. The retract case follows easily
from the cell complex case, so we assume that $F$ is a cell complex.  This means that $0 \rightarrow F$ is a relative
cell complex, so that there is an ordinal $\lambda$ and a transfinite diagram $X: \lambda \rightarrow [\mcB^\op, \mcC]$
such that
\begin{enumerate}
\item[(i)]  $X_0 = 0$,
\item[(ii)]  $\ds \colim_{\alpha} X_\alpha = F$,
\item[(iii)] $0 \rightarrow F$ is the composition of $X$, and
\item[(iv)] for each ordinal $\alpha < \lambda$, the map $X_\alpha \rightarrow X_{\alpha + 1}$ fits into a pushout square
\[
\begindc{0}[3]
\obj(-9,8)[1]{$F_1$} \obj(9,8)[2]{$F_2$} \obj(-9,-8)[3]{$X_\alpha$}
\obj(9,-8)[4]{$X_{\alpha + 1}$} \mor{1}{2}{$i$} \mor{3}{4}{}
\mor{1}{3}{} \mor{2}{4}{}
\enddc
\]
where $F_1$ and $F_2$ are free functors and $i: F_1 \rightarrow F_2$ is a generating cofibration in the category of diagrams.
\end{enumerate}

For each ordinal $\alpha < \lambda$ we have a commutative square
\[
\begindc{0}[3]
\obj(-11,8)[1]{$X_\alpha$} \obj(11,8)[2]{$\mcR(\mcL X_\alpha)$} \obj(-11,-8)[3]{$F$}
\obj(11,-8)[4]{$\mcR(\mcL F)$} \mor{1}{2}{} \mor{3}{4}{$\eta_F$}
\mor{1}{3}{} \mor{2}{4}{}
\obj(17,-9){.}
\enddc
\]
Taking homotopy colimits gives a commutative square
\[
\begindc{0}[3]
\obj(-18,9)[1]{$\ds \hocolim_\alpha X_\alpha$} \obj(18,9)[2]{$\ds \hocolim_\alpha \mcR(\mcL X_\alpha)$} \obj(-18,-9)[3]{$F$}
\obj(18,-9)[4]{$\mcR(\mcL F)$} \mor{1}{2}{(C)} \mor{3}{4}{$\gamma(\eta_F)$}
\mor{1}{3}{(A)}[\atright, \solidarrow] \mor{2}{4}{(B)}
\enddc
\]
in the homotopy category, where $\gamma: \mcC \rightarrow \ho(\mcC)$ is the canonical localization.  We claim that the maps
labelled (A), (B) and (C) are isomorphisms.  This would make $\gamma(\eta_F)$ an isomorphism, so that $\eta_F$ is then a weak
equivalence.  In fact, (A) is readily seen to be an isomorphism:  since $X$ is necessarily a cofibrant diagram, the map
$\ds \ilhocolim_\alpha X_\alpha \rightarrow \ilcolim_\alpha X_\alpha = F$ is an isomorphism.

Next, we show that map (B) is an isomorphism in the homotopy category.  The key fact
is that $\mcL X$ is again a cofibrant diagram, so that $\ds \ilhocolim_\alpha \mcL X_\alpha \cong \ilcolim_\alpha
\mcL X_\alpha \cong \mcL F$.  Invoking Colimit-Product Coherence, upon evaluation at an object $b$ of
$\mcB$ we have a sequence of natural isomorphisms
\begin{eqnarray*}
\hocolim_\alpha \mcR(\mcL X_\alpha)(b) & = & \hocolim_\alpha \prod_{i \in \sk(\mcI \downarrow b)} \mcL X_\alpha(\dom(i))\\
 & = & \prod_{i \in \sk(\mcI \downarrow b)} \hocolim_\alpha \mcL X_\alpha(\dom(i))\\
 & = & \prod_{i \in \sk(\mcI \downarrow b)} \colim_\alpha \mcL X_\alpha(\dom(i))\\
 & = & \mcR(\mcL F)(b).
\end{eqnarray*}
Hence map (B) is an isomorphism.

All that remains to be shown is that map (C) is an isomorphism.  It suffices to prove that $X_\beta \rightarrow
\mcR(\mcL X_\beta)$ is a weak equivalence for each ordinal $\beta < \lambda$, and for this we argue by transfinite induction.  Fix
an ordinal $\beta$ and assume that $X_\alpha \rightarrow \mcR(\mcL X_\alpha)$ is a weak equivalence for each $\alpha <
\beta$.  There are two cases:  $\beta$ is a limit ordinal, or it is not.

In one case, $\beta$ is not a limit ordinal, so that $\beta = \alpha + 1$ for some $\alpha$.  By hypothesis, $X_\alpha
\rightarrow \mcR(\mcL X_\alpha)$ is a weak equivalence.  By examining statement (iv) above, we see that Proposition
\ref{pushout_equiv} implies that $X_\beta \rightarrow \mcR(\mcL X_\beta)$ is indeed a weak equivalence.

In the last case, $\beta$ is a limit ordinal, so that $\ds \ilcolim_{\alpha < \beta} X_\alpha \rightarrow X_\beta$ is
an isomorphism in $\mcC$.  An argument exactly like that above for map (B) shows that $\ds \ilhocolim_{\alpha < \beta} \mcR(\mcL X_\alpha)
\rightarrow \mcR(\mcL X_\beta)$ is an isomorphism in $\ho(\mcC)$.  Moreover, the inductive hypothesis implies that
the map $\ds \ilhocolim_{\alpha < \beta} X_\alpha \rightarrow \ilhocolim_{\alpha < \beta} \mcR(\mcL X_\alpha)$ is an isomorphism
as well.  We then have a commutative square
\[
\begindc{0}[3]
\obj(-18,9)[1]{$\ds \hocolim_{\alpha < \beta} X_\alpha$} \obj(18,9)[2]{$\ds \hocolim_{\alpha < \beta} \mcR(\mcL X_\alpha)$} \obj(-18,-9)[3]{$\ds X_\beta = \colim_{\alpha < \beta} X_\alpha$}
\obj(18,-9)[4]{$\mcR(\mcL X_\beta)$} \mor{1}{2}{$\simeq$} \mor{3}{4}{}
\mor{1}{3}{$\simeq$}[\atright, \solidarrow] \mor{2}{4}{$\simeq$}
\enddc
\]
in $\ho(\mcC)$ with isomorphisms as indicated.  Thus, $X_\beta \rightarrow \mcR(\mcL X_\beta)$ is a weak equivalence.  Transfinite
induction now shows that map (C) is an isomorphism, as desired.
\end{proof}

\begin{thm}\label{main_thm}
Suppose that $(\mcB, \mcA)$ is a conjugate pair of small categories and that $\mcC$ is a semi-stable model category.  Then the adjoint
pair
\[
\begindc{0}[3]
\obj(-12,0)[1]{$[\mcB^\op, \mcC]$}
\obj(12,0)[2]{$[\mcA^\op, \mcC]$}
\obj(-8,1)[3]{} \obj(8,1)[4]{} \obj(-8,-1)[5]{} \obj(8,-1)[6]{} \mor{3}{4}{$\mcL$} \mor{6}{5}{$\mcR$}
\enddc
\]
associated to the regular bimodule $\mcU_+: \mcA^\op \times \mcB \rightarrow \Sets_*$ is a Quillen equivalence.
\end{thm}

\begin{proof}
By the well-known criteria, it suffices to show that for each cofibrant functor $F: \mcB^\op \rightarrow \mcC$
and each fibrant functor $G: \mcA^\op \rightarrow \mcC$, a map $\tau: \mcL F \rightarrow G$ is a
weak equivalence if and only if its adjoint $\tau^\#: F \rightarrow \mcR G$ is as well.  The crucial observation
is that the adjoint $\tau^\#$ is simply the composite
\[
\tau^\#: F \stackrel{\eta_F}{\longrightarrow} \mcR(\mcL F) \stackrel{\mcR
\tau}{\longrightarrow} \mcR G.
\]
As $F$ is cofibrant, $\eta_F$ is then a weak equivalence.  Hence $\tau^\#$ is a weak equivalence if and only if
$\mcR\tau$ is, and the result follows immediately from Corollary~\ref{detects}.
\end{proof}

In the case that $\mcC$ is abelian and given the trivial model structure, the theorem reduces to the following.

\begin{cor}
Suppose that $\mcC$ is a complete and cocomplete abelian category.  Given any conjugate pair $(\mcB, \mcA)$, the adjoint
pair
\[
\begindc{0}[3]
\obj(-12,0)[1]{$[\mcB^\op, \mcC]$}
\obj(12,0)[2]{$[\mcA^\op, \mcC]$}
\obj(-8,1)[3]{} \obj(8,1)[4]{} \obj(-8,-1)[5]{} \obj(8,-1)[6]{} \mor{3}{4}{$\mcL$} \mor{6}{5}{$\mcR$}
\enddc
\]
associated to the regular bimodule is an equivalence of categories.
\end{cor}

\begin{example}
Applying the theorem to Example~\ref{idem_example} gives the well-known fact that idempotents split in stable homotopy and
abelian categories.  This is not new of course, but it's satisfying to see this manifest itself here.
\end{example}

\begin{example}
Let $\mcC$ be any complete and cocomplete abelian category.  Applying the theorem to the conjugate pair $(\Gamma, \mcE)$,
we recover Pirashvili's first main result in \cite{TP}.  Taking $\mcC$ to be a stable model category satisfying our
technical assumptions, we confirm the conjecture that started this project:  the categories of functors indexed by $\Gamma^\op$
and $\mcE^\op$ are in fact Quillen equivalent when taking values in a stable model category.
\end{example}


In conclusion, we remark that this is not a complete Morita theory for semi-stable model categories of diagrams.  The best
possible result would be a complete characterization of the pairs $(\mcB, \mcA)$ yielding a Quillen equivalence.  We have
given sufficient---but not necessary---conditions for the existence of such pairs of categories.  Moreover,
classical Morita equivalence is characterized by finitely generated projective generators.  In our case, the regular bimodule is more than projective:  it
is free.  If the analogy is to be believed, one would think there must be a suitable projective version of this development.


\end{document}